%% file: height-1-conditions.tex
\documentclass[reqno]{amsart}
\usepackage{mathrsfs}

\usepackage[T1]{fontenc}
\usepackage[utf8]{inputenc}

\usepackage{tikz}

\definecolor{green}{cmyk}{.51,0,.91,.34}
\usepackage[colorlinks=true,
  linkcolor=black,
  citecolor=green,
  bookmarks]{hyperref}

\newcommand\changed[1]{#1}

\newtheorem{theorem}{Theorem}[section]
\newtheorem{lemma}[theorem]{Lemma}
\newtheorem{corollary}[theorem]{Corollary}

\newtheorem{conjecture}[theorem]{Conjecture}

\theoremstyle{definition}
\newtheorem{definition}[theorem]{Definition}

\newtheorem{question}[theorem]{Question}
\newtheorem{construction}[theorem]{Construction}

\theoremstyle{remark}
\newtheorem{remark}[theorem]{Remark}

\let\rel\mathbb
\let\clo\mathscr
\newcommand{\Proj}{{\clo P}}
\newcommand{\proj}{\operatorname{pr}}

\let\equals\approx
\newcommand{\NP}{\text{NP}}
\newcommand{\Ptime}{\text{P}}

\DeclareMathOperator{\Pol}{Pol}
\DeclareMathOperator{\CSP}{CSP}
\DeclareMathOperator{\Aut}{Aut}
\DeclareMathOperator{\Forb}{Forb}
\DeclareMathOperator{\CSS}{CSS}

\newcommand{\bB}{{\rel B}}
\newcommand{\aA}{{\rel A}}
\newcommand{\cC}{{\rel C}}
\newcommand{\sS}{{\rel S}}
\newcommand{\sSa}{{\rel S}}
\newcommand{\gG}{{\rel G}}
\newcommand{\hH}{{\rel H}}

\begin{document}

  \author[Bodirsky]{Manuel Bodirsky}
  \address{Institute of Algebra, Technische Universit\"at Dresden, Dresden, Germany}
  \email{manuel.bodirsky@tu-dresden.de}
  
  \author[Mottet]{Antoine Mottet}
  \address{Department of Algebra, Charles University in Prague, Czech Republic}
  \email{mottet@karlin.mff.cuni.cz}
  
  \author[Ol\v{s}\'ak]{Miroslav Ol\v{s}\'ak}
  \address{Department of Algebra, Charles University in Prague, Czech Republic}
  \email{mirek@olsak.net}

  \author[Opr\v{s}al]{Jakub Opr\v{s}al}
  \address{Department of Computer Science, Durham University, Durham, UK}
  \email{jakub.oprsal@durham.ac.uk}

  \author[Pinsker]{Michael Pinsker}
  \address{Institute of Discrete Mathematics and Geometry, Technische Universit\"at Wien, Austria, and Department of Algebra, Charles University in Prague, Czech Republic}
  \email{marula@gmx.at}

  \author[Willard]{Ross Willard}
  \address{Department of Pure Mathematics, University of Waterloo, Waterloo, ON, Canada}
  \email{ross.willard@uwaterloo.ca}

  \title[$\omega$-categorical structures avoiding height~1 identities]{$\omega$-categorical structures avoiding height~1 identities}

  \hypersetup{
    pdftitle={ω-categorical structures avoiding height 1 identities},
    pdfauthor={Manuel Bodirsky, Antoine Mottet, Miroslav Olsak, Jakub Oprsal, Michael Pinsker, and Ross Willard}%
  }

  \thanks{%
    Manuel Bodirsky and Jakub Opr\v{s}al were supported by the European Research Council (ERC) under the European Union's Horizon 2020 research and innovation programme (Grant Agreement No 681988, CSP-Infinity).
    Jakub Opr\v{s}al has also received funding from the UK EPSRC (grant No EP/R034516/1).
    Antoine Mottet has received funding from the European Research Council (ERC) under the European Union's Horizon 2020 research and innovation programme (Grant Agreement No 771005, CoCoSym).
    Miroslav Ol\v{s}\'{a}k and Michael Pinsker have received funding  from the Czech Science Foundation (grant No 18-20123S).
    Michael Pinsker has received funding from the Austrian Science Fund (FWF) through project No {P32337}.
    Ross Willard was supported by the Natural Sciences and Engineering Research Council of Canada.
  }
  \thanks{%
    A conference version of this material appeared at the Thirty-Fourth Annual ACM/IEEE Symposium on Logic in Computer Science (LICS) 2019~\cite{TopologyIsRelevant}.
  }


  \begin{abstract}
  The algebraic dichotomy conjecture for Constraint Satisfaction Problems (CSPs) of reducts of (infinite) finitely bounded homogeneous structures states that such CSPs are polynomial-time tractable if the model-complete core of the template has a pseudo-Siggers polymorphism, and NP-complete otherwise.

  One of the important questions related to \changed{the dichotomy conjecture} is whether, similarly to the case of finite structures, the condition of having a pseudo-Siggers polymorphism can be replaced by the condition of having polymorphisms satisfying a fixed set of identities of height~1, i.e., identities which do not contain any nesting of functional symbols. We provide a negative answer to this question by constructing for each non-trivial set of height~1 identities a structure {within the range of the conjecture} whose polymorphisms do not satisfy these identities, but whose CSP is tractable nevertheless.

  An equivalent formulation of the dichotomy conjecture characterizes tractability of the CSP via the local satisfaction of non-trivial height~1 identities by polymorphisms of the structure. We show that local satisfaction and global satisfaction of non-trivial height~1 identities differ for $\omega$-categorical structures with less than doubly exponential orbit growth, thereby resolving one of the main open problems in the algebraic theory of such structures.
  \end{abstract}

  \maketitle

\input{1_introduction}

\input{2_preliminaries}
\input{3_sigma_G}
\input{5_no_strong}

\input{6_topology}

%
  \bibliographystyle{alpha}
  \bibliography{local,global}

  \appendix
\input{a_construction}

\end{document}

%% file: 1_introduction.tex
\section{Introduction}

Many computational problems in theoretical computer science can be phrased as  \emph{constraint satisfaction problems (CSPs)}: in such a~problem, we are given a~finite set of variables and a~finite set of constraints that are imposed on the variables, and the task is to find values for the variables that satisfy all the given constraints. The computational complexity of a~CSP depends on the language that we allow when formulating the constraints in the input. By appropriately choosing this language, many computational problems in optimisation, artificial intelligence, computational biology, verification, and many other areas can be precisely expressed as a~CSP.

Formally, we fix a~{relational} structure $\rel B$ (also called the \emph{template} or \emph{constraint language}).  The problem $\CSP(\rel B)$ is the computational problem of deciding whether a~given conjunction of atomic formulas over the signature of $\rel B$ is satisfiable in $\rel B$.  For example, if the domain of $\rel B$ is the Boolean domain $\{0,1\}$, and $\rel B$ contains all binary Boolean relations, then $\CSP(\rel B)$ is precisely the 2-\textsc{Sat} problem, which can be solved in polynomial time, and if the structure $\rel B$ is the complete graph $\rel K_3$ on three vertices (without loops), then $\CSP(\rel B)$ is precisely the graph 3-coloring problem, which is \NP-complete.  Note that it is not necessary for this definition that the domain of $\rel B$ is finite, and indeed, many problems can only be expressed {if} the domain of $\rel B$ is infinite.  For example the satisfiability of a~system of polynomial equations over the rational numbers can be formulated using a~structure $\rel B$ whose domain is the rationals, but certainly not with a~structure $\rel B$ that has a~finite domain.

The class of CSPs is a~large class which allows for a~uniform mathematical approach to the question which interests us for computational problems in general:
\begin{center}\emph{What kind of structure makes a~problem easy (i.e., polynomial-time tractable), and what makes such a~problem hard (i.e., \NP-hard)?}\end{center}
Evidence for the possibility of a~clear structural characterization of tractability within the realm of large classes of CSPs has been found in the fact that finite-domain CSPs exhibit a~\Ptime/\NP-complete dichotomy, that is, \changed{for every finite structure $\rel B$, the problem} $\CSP(\rel B)$ is in \Ptime\ or \NP-complete. This was conjectured by Feder and Vardi \cite{FV98}, and recently proved by Bulatov \cite{Bul17} and, independently, by Zhuk \cite{Zhu17}.
Both proofs rely on the universal-algebraic approach and recent developments in universal algebra. In fact, they prove a~strengthening of the conjecture which in addition provides a~precise condition that {characterizes} \NP-completeness of a~finite-domain CSP. This strengthening provided by Bulatov, Jeavons, and Krokhin~\cite{BJK05} uses algebraic language, in particular the notion of \emph{polymorphisms}, which are~(structure preserving) finitary functions on a~structure $\rel B$.
Such functions can be viewed as `higher-order symmetries', and in particular they form a~certain generalization of the automorphisms of $\rel B$.
The essence of the algebraic approach is that the complexity of $\CSP(\rel B)$ is determined up to log-space reductions by the polymorphisms of $\rel B$: {the absence of ``interesting"} polymorphisms implies hardness of the CSP, while {their presence} {makes certain algorithms applicable to the CSP}.

Before we move to the infinite case, let us first describe the situation in the finite case. We denote by $\Pol(\rel B)$ the set of all polymorphisms of $\rel B$, and by $\Proj$ the set of projections, i.e., trivial polymorphisms {on the set $\{0,1\}$} (these are precisely the polymorphisms of 3-\textsc{Sat}).  Using results of Siggers \cite{Sig10} {and from~\cite{BOP18}}, the finite-domain CSP dichotomy can then be formulated as follows (see Section~\ref{sect:prelims} for the definitions of the concepts that appear in the statement).
\begin{theorem}[Bulatov-Zhuk \cite{Bul17,Zhu17}]\label{thm:bulatov-zhuk}
Let $\rel B$ be a~finite structure.
Exactly one of the following holds:
\begin{enumerate}
	\item There exists a~minion homomorphism $\Pol(\rel B)\to \Proj$, and $\CSP(\rel B)$ is \NP-complete,
	\item $\Pol(\rel B)$ contains a~function $s$ satisfying the \changed{Siggers} identity
        \begin{equation}
          \forall x,y,z\in B\;\; s(x,y,x,z,y,z) = s(y,x,z,x,z,y),
            \tag{$\diamondsuit$}\label{eq:siggers}
        \end{equation}
        and $\CSP(\rel B)$ is in \Ptime.
\end{enumerate}
In particular, $\CSP(\rel B)$ is in \Ptime\ or \NP-complete.
\end{theorem}

One astonishing fact is that the condition for tractability in this dichotomy has an elegant formulation using a~single function satisfying a~single {non-nested ({\emph{height~1}})} identity.

For templates with an infinite domain, it is known that no such dichotomy exists in general and that CSPs exhaust all possible complexity classes, up to polynomial-time Turing reductions~\cite{BG08}.
However, large classes of infinite templates have been proved to exhibit a~\Ptime/\NP-complete dichotomy. One of the largest and most robust classes that have been conjectured
to have such a~dichotomy is the class of so-called \emph{first-order reducts of finitely bounded homogeneous structures}.
This class is important for several reasons:
\begin{itemize}
\item It is a~vast generalisation 
of the class of finite structures \changed{for which} it is possible to investigate deep questions about the nature of computation. 
Many problems studied e.g.~in temporal and spatial reasoning can be formulated as CSPs for such structures \cite{BJ17}. 
\item \changed{All the structures in this class are \emph{$\omega$-categorical}, i.e., they have an oligomorphic automorphism group.} Thus, the algebraic methods of finite-domain constraint satisfaction can still be used in this class~\cite{BodirskyNesetrilJLC,Topo-Birk}, and due to its relative tameness \changed{the class provides} an important framework where the tractability of large classes of computational problems can be tied to algebraic and topological properties of mathematical objects.
\end{itemize}
The dichotomy conjecture for \changed{such} structures has been verified in numerous special cases, for example for all CSPs in the complexity class MMSNP \cite{MMSNP}; also see~\cite{tcsps-journal,BodPin-Schaefer-both,Phylo-Complexity,posetCSP16,BMPP16}.
There are various equivalent formulations of the infinite-domain tractability conjecture originally formulated in \cite{BPP-projective-homomorphisms} (Conjecture~\ref{conj:tract}).  The most recent one, proposed by Barto, Opr\v{s}al, and Pinsker~\cite{BOP18} and later proved to be equivalent to the original one~\cite{BKOPP17a,BKOPP17}, is now considered the most satisfactory formulation both esthetically and practically:

\begin{conjecture}\label{conj:tract2}
Let $\rel B$ be a~reduct of a~finitely bounded homogeneous structure. Exactly one of the following holds:
\begin{enumerate}
  \item There exists a~uniformly continuous minion homomorphism from $\Pol(\rel B)$ to $\Proj$, and $\CSP(\rel B)$ is \NP-complete. 
  \item $\Pol(\rel B)$ does not have a~uniformly continuous minion homomorphism to $\Proj$, and $\CSP(\rel B)$ is in \Ptime.
\end{enumerate}
\end{conjecture}

Here, uniform continuity is meant with respect to the {natural uniformity which induces the \emph{pointwise convergence topology} on the space of all finitary functions on a set.}
It is known that if there exists a~uniformly continuous minion homomorphism $\Pol(\rel B)\to\Proj$, then $\CSP(\rel B)$ is \NP-complete~\cite{BOP18}.

There are two major differences between the above conjecture and the finite-domain CSP dichotomy as phrased in Theorem~\ref{thm:bulatov-zhuk}.
First of all, there is topological content in Conjecture~\ref{conj:tract2}.
This topological nature can loosely be explained in the following terms: the non-existence of a~uniformly continuous minion homomorphism $\Pol(\rel B)\to\Proj$
can be characterised by the fact that non-trivial {height~1} identities are satisfied {in $\Pol(\rel B)$} on every finite subset of $\rel B$, while the non-existence of a~minion homomorphism $\Pol(\rel B)\to\Proj$
is characterised by the fact that some non-trivial {height~1} identities are satisfied {in $\Pol(\rel B)$} on the whole structure $\rel B$.
This local/global distinction evidently only arises when $\rel B$ is an infinite structure, and can be understood as one of the major obstacles \changed{to} solving Conjecture~\ref{conj:tract2}.

Second, even when ignoring the topology, it is not known whether \changed{clause (2) in Conjecture~\ref{conj:tract2}} can be expressed by the satisfaction of some fixed height~1 identities in $\Pol(\rel B)$ as is the case in Theorem~\ref{thm:bulatov-zhuk}(2). This naturally raises the following questions, which were also asked in~\cite{BartoPinskerDichotomy, BOP18, Topo}:
\begin{enumerate}
  \item Does the existence of a~minion homomorphism $\Pol(\rel B)$ to $\Proj$ imply the existence of a~uniformly continuous one? In other words, can the requirement of uniform continuity in Conjecture~\ref{conj:tract2} be dropped?
  \item Can the non-existence of a~minion homomorphism to $\Proj$ be replaced by a~statement positing that some fixed set of height 1 identities holds in $\Pol(\rel B)$?
\end{enumerate}
We note that a~positive answer to the second question would have \changed{consequences which would} make a~positive answer to the first question more likely.
Moreover, the corresponding natural questions were asked about the historically first conjecture (see~\cite{BPP-projective-homomorphisms} and Conjecture~\ref{conj:tract})
and were proved to have positive answers~\cite{BartoPinskerDichotomy, Topo}, thus showing that topology is irrelevant in that formulation of the conjecture. 

The second question is purely algebraic, and therefore of interest to universal algebra as well. Similar questions have been asked about various properties of algebras, e.g.\ \cite{Tay88,Ols17}. One notable open problem in this field is whether the algebraic condition describing structures whose CSP can be solved by a~Datalog program \cite{BK14} can be described by a~single fixed set of {height~1} identities as well. As in Theorem~\ref{thm:bulatov-zhuk}, it is known that there is such a~set of {height~1} identities when we restrict to finite domains.

\subsection*{Contributions}

In the present paper, we give a~negative answer to the second question, proving that no system of height~1 identities (also called \emph{height 1 condition}) can be used as a~replacement for the condition in the second item of Conjecture~\ref{conj:tract2}. Our result is formalized as follows:
\begin{theorem} \label{thm:main-1}
  For every non-trivial height 1 condition $\Sigma$ there exists a~structure $\rel B$ such that 
  \begin{itemize}
  \item $\rel B$ is a~first-order reduct of a~finitely bounded homogeneous structure;
  \item $\Pol(\rel B)$ does not satisfy $\Sigma$;
  \item $\Pol(\rel B)$ satisfies some other non-trivial height 1 condition (consequently, there is no minion homomorphism to $\Proj$);
  \item $\CSP(\rel B)$ is in \Ptime.
 \end{itemize}
 \end{theorem}
{We remark that the CSPs of the structures constructed in the proof of this theorem are in~P because they can be solved by ``local checking", and hence by a Datalog program. Consequently, we also show that those CSPs which are solvable in this manner cannot be described by the satisfaction of a fixed set of height~1 identities.}

On the other hand, we construct an infinite chain of \changed{successively weaker} systems of height~1 identities { that can be used to describe the non-existence of a~minion homomorphism to $\Proj$. Decreasing chains of sets of identities, called \emph{Mal'cev conditions}~\cite{HobbyMcKenzie}, are commonly studied in universal algebra since many interesting properties of algebras can be expressed by such a~condition, e.g.\ congruence distributivity \cite{Jon67}.  The members of our chain are  naturally \changed{encoded by} finite graphs, and generalize the single identity of Theorem~\ref{thm:bulatov-zhuk}. In both proofs of the latter theorem, such a decreasing chain of height~1 conditions was sufficient to prove the dichotomy (i.e., the authors do not rely on the satisfaction of the Siggers identity~\ref{eq:siggers}). We will provide a formal statement in  Theorem~\ref{thm:chain}.}

This still leaves the possibility that Question~(1) above has a~positive answer. 
\changed{We provide a negative answer in a wider context by combining the structures from Theorem~\ref{thm:main-1} into a single structure $\rel S$.
The structure $\rel S$ is not a~first-order reduct of a~finitely bounded homogeneous structure, and therefore does not \changed{lie within} the scope of Conjecture~\ref{conj:tract2}.
However, it is $\omega$-categorical and has slow orbit growth, which shows that the techniques that were used to show that the two infinite-domain dichotomy conjectures are equivalent~\cite{BKOPP17a,BKOPP17} cannot be employed to remove the topological considerations from Conjecture~\ref{conj:tract2}.}

\begin{theorem} \label{thm:main-3}
There exists a~structure $\rel S$ with the following properties.
\begin{enumerate}
\item $\rel S$ is an $\omega$-categorical structure with less than doubly exponential orbit growth.
  \item $\Pol(\rel S)$ has a~minion homomorphism to $\Proj$.
  \item $\Pol(\rel S)$ has no uniformly continuous minion homomorphism to $\Proj$.
\end{enumerate}
\end{theorem}
\changed{In the proof of~(3), we use the fact that $\rel S$ has slow orbit growth (by construction), and a recent result from~\cite{BKOPP17a,BKOPP17} characterizing~(3) via so-called \emph{pseudo-Siggers} operations in structures with slow orbit growth.}

\subsection*{Outline}
{This article is organized as follows. After giving the necessary precise definitions and notation in Section~\ref{sect:prelims}, we explore height~1 conditions \changed{induced by} graphs in Section~\ref{sec:sigma-g}, and construct the weakest height~1 Mal'cev condition mentioned above. In Section~\ref{sect:noweakest}, we then use the height~1 conditions which appear in our Mal'cev condition in order to prove Theorem~\ref{thm:main-1}. Finally,  in Section~\ref{sec:topology} we combine the structures thus constructed into a single structure and prove Theorem~\ref{thm:main-3}.}

%% file: 2_preliminaries.tex
\section{Notation and definitions}
  \label{sect:prelims}

We recall some basic notions from the algebraic approach to CSPs as well as \changed{notions} from model theory. We refer to~\cite{BKW17} and~\cite{Hod97} for more detailed introductions to these topics.

\subsection{Structures, polymorphisms}

A \emph{signature} is a {family} 
$\sigma=(R_i)_{i\in I}$ of symbols, where each symbol is associated with a natural number called its \emph{arity}.
A $\sigma$-structure $\rel A$ is a \changed{pair} $(A,(R^{\rel A}_i)_{i\in I})$ consisting of a set (the \emph{domain}) together with a 
{family} 
$(R^{\rel A}_i)_{i\in I}$ of relations on $A$,
where for all $i\in I$ the relation $R^{\rel A}_i$ has the arity specified by $\sigma$.

A \emph{graph} is a~structure with a~single binary symmetric relation; in particular, in this {article} 
all graphs are undirected.
We denote by $\rel K_3$ the complete graph on three vertices.

Let $\rel A, \rel B$ be two structures with the same signature (e.g., two graphs). A map $h\colon A\to B$ is a~\emph{homomorphism} from $\aA$ to $\bB$ if it preserves all relations, i.e., for all $i\in I$,
\begin{equation}\label{al:homo}
  \text{if } (a_1,\dots,a_k)\in R^{\rel A}_i \text{, then } (h(a_1),\dots,h(a_k)) \in R^{\rel B}_i.\tag{$\clubsuit$}
\end{equation}
Two structures $\rel A$ and $\rel B$ are \emph{homomorphically equivalent} if there exist homomorphisms from $\rel A$ to $\rel B$ and from $\rel B$ to $\aA$. An \emph{embedding} of $\aA$ into $\bB$ is an injective homomorphism from $\aA$ to $\bB$ such that the implication in (\ref{al:homo}) is an equivalence.

For $n\geq 1$, we define the $n$-th power of a~structure $\rel A$ to be the structure $\rel A^n$ with same signature, whose domain is $A^n$, and such that for all $i\in I$, a tuple $(\overline{a}^1,\dots,\overline{a}^k)$ of $n$-tuples is contained in $R^{\rel A^n}_i$ if, and only if, it is contained in $R^{\rel A}_i$ componentwise, i.e., $(a^1_j,\dots,a^k_j)\in R^{\rel A}_i$ for all $1\leq j\leq  n$.
For graphs, this power is often called \emph{tensor power} since the adjacency matrix of the power is a~tensor power of the adjacency matrix of the original graph.

A~\emph{polymorphism} of a structure $\rel A$ is a~homomorphism from $\rel A^n$ to $\rel A$, for some $n\geq 1$. We write $\Pol(\rel A)$ for the set of all polymorphisms of a~structure $\rel A$. An \emph{endomorphism} of $\aA$ is a homomorphism from $\aA$ to $\aA$, i.e., a unary polymorphism of $\aA$. An \emph{automorphism} of $\aA$ is a bijective embedding of $\aA$ into~$\aA$. 

\subsection{Clones and height 1 conditions}

Let $A$ be a set. A \emph{clone} is a set $\clo A$ of finitary operations on $A$ satisfying the following conditions:
\begin{itemize}
	\item for all $n\geq 1$ and $1\leq i\leq n$, the $i$-th $n$-ary projection $\proj_i^n\colon (x_1,\dots,x_n)\mapsto x_i$
	is a function in $\clo A$;
	\item for all $n$-ary $f\in\clo A$ and all $m$-ary $g_1,\dots,g_n\in\clo A$, the composition $f\circ (g_1,\dots,g_n)\colon (x_1,\dots,x_m)\mapsto f(g_1(x_1,\dots,x_m),\dots,g_n(x_1,\dots,x_m))$
	is in $\clo A$.
\end{itemize}
For $n\geq 1$, we denote by $\clo A^{(n)}$ the set of $n$-ary functions in $\clo A$. We moreover write $\Proj$ for the clone on the set $\{0,1\}$ that consists only of projections. It is easy to see that for every structure $\rel A$ the set $\Pol(\rel A)$ is a~clone.

A~\emph{height 1 identity} is a statement of the form
\[
  \forall x_1,\dots,x_r\;\; f(x_{\pi(1)},\dots,x_{\pi(n)}) = g(x_{\rho(1)},\dots,x_{\rho(m)})
\]
where $f,g$ are function symbols, and
\[\pi\colon \{1,\ldots,n\}\to\{1,\ldots,r\},\;\; \rho\colon \{1,\ldots,m\}\to\{1,\ldots,r\}\] are any functions. We also write \[f(x_{\pi(1)},\dots,x_{\pi(n)}) \approx g(x_{\rho(1)},\dots,x_{\rho(m)})\] for such an identity, omitting the universal quantification. An example of {a height~1} 
identity is the Siggers identity (\ref{eq:siggers}).

A~\emph{height 1  condition} is a~finite set $\Sigma$ of height 1 identities (where several identities can \changed{involve} the same function symbol). Such a~condition is said to be satisfied in a~set of functions $\clo A$ (e.g.\ a~clone) if for each function symbol $f$ appearing in $\Sigma$, there exists a~function $f^{\clo A} \in \clo A$ of the corresponding arity such that every identity in $\Sigma$ becomes a true statement when the symbols of $\Sigma$ are instantiated by their counterparts in $\clo A$. In that case, we   say that \emph{$\clo A$ satisfies $\Sigma$}.

The notion of satisfaction in a~clone \changed{leads to} a~natural quasi-order on height~1 conditions: if every clone \changed{that satisfies} $\Sigma$ also satisfies $\Sigma'$, then we say that $\Sigma$ \emph{implies} $\Sigma'$ (or that $\Sigma'$ is \emph{weaker} than $\Sigma$, or that $\Sigma$ is \emph{stronger} than $\Sigma'$).  Two conditions $\Sigma$ and $\Sigma'$ are {\emph{equivalent}} 
\changed{if they belong to the same equivalence class of the quasi-order, i.e., if every clone satisfies either both or neither of $\Sigma$ and $\Sigma'$.}
The \emph{strictly decreasing chain} {mentioned after Theorem~\ref{thm:main-1}} 
is to be understood in this quasi-order, i.e., it is a sequence of conditions of strictly decreasing strength.

A~height~1 condition is \emph{trivial} if it is satisfied in every clone, or equivalently, if it holds in $\Proj$, or again equivalently, if it is implied by any other height~1 condition.

\subsection{Minion homomorphisms}
Let $\clo A$, $\clo B$ be two clones. We say that a mapping $\xi\colon \clo A\to \clo B$ is a~\emph{minion homomorphism}\footnote{{This notion was introduced as \emph{h1 clone homomorphism} in~\cite{BOP18}.} A~\emph{minion} is a certain type of abstract algebraic structure and minion homomorphisms as introduced here correspond to the natural maps between minions. The definition of a~minion \cite[Definition 2.20]{BBKO19} is irrelevant for our purposes so we omit it.}
if it preserves arities, and for all $f\in \clo A$ of arity $n$ and all $\pi\colon \{1,\ldots,n\} \to \{1,\ldots,k\}$ we have
\[
  \xi(f(x_{\pi(1)},\dots,x_{\pi(n)})) = \xi(f)(x_{\pi(1)},\dots,x_{\pi(n)}).
\]

Note that a~minion homomorphism preserves height 1 identities, and therefore also height 1 conditions, i.e., if there is a~minion homomorphism from $\clo A$ to $\clo B$ and $\Sigma$ is a~height 1 condition such that $\clo A$ satisfies $\Sigma$, then also $\clo B$ satisfies $\Sigma$.
In particular, if there exists a~minion homomorphism from $\clo A$ to the projection clone $\Proj$, then $\clo A$ only satisfies trivial height~1 conditions.
The converse of the latter statement also holds as can be proved by a~compactness argument; it also follows from Lemma~\ref{lem:1.2} of the present article.

When $\clo A$ and $\clo B$ are clones, then {an arity-preserving} map $\xi\colon\clo A\to\clo B$ is called \emph{uniformly continuous}\footnote{Clones can be naturally endowed with a uniform structure. The notion of uniform continuity corresponding to this uniform structure agrees with the definition given here; the interested reader will find details in~\cite{uniformbirkhoff, BOP18}.} if for every $n\geq 1$ {and every finite set $F\subseteq B^n$ there exists a~finite set $S\subseteq A^n$ such that $f|_{S} = g|_{S}$ implies $\xi(f)|_{F} = \xi(g)|_{F}$ for all $f,g\in\clo A^{(n)}$}. The non-existence of uniformly continuous minion homomorphisms from a clone $\clo A$ to $\Proj$ is equivalent to  non-trivial height~1 conditions being satisfied in $\clo A$ \emph{locally}, in the following sense.
Let $\Sigma$ be a height~1 condition, and let $S\subseteq A$. We say that $\clo A$ \emph{satisfies $\Sigma$ on $S$} if it satisfies $\Sigma$ when the quantified variables in the identities of $\Sigma$ only range over $S$ (rather than $A$); i.e., the identities of $\Sigma$ are replaced by formulas of the form
\[
  \forall x_1,\dots,x_r \in S\;\; f(x_{\pi(1)},\dots,x_{\pi(n)}) = g(x_{\rho(1)},\dots,x_{\rho(m)})\; .
\]
Then there is no uniformly continuous minion homomorphism from $\clo A$ to $\Proj$ if and only if  for every finite subset $S$ of $A$ there exists a non-trivial height~1 condition that is satisfied by $\clo A$ on $S$ (again, this can be proved by a~compactness argument).

\subsection{Logic and model theory}

The set of automorphisms of $\aA$ forms a~group denoted by $\Aut(\aA)$. For all $k\geq 1$, this group acts naturally on $A^k$ by $\alpha\cdot(a_1,\dots,a_k) := (\alpha(a_1),\dots,\alpha(a_k))$.
An \emph{orbit} is a~set of the form $\{\alpha\cdot \overline a\mid \alpha\in\Aut(\aA)\}$ for some $\overline a\in A^k$.
{The function $f$ which assigns to every $k\geq 1$ the cardinality of the set of orbits of $\Aut(\aA)$ on $k$-tuples is non-decreasing, and its growth is an interesting measure of the combinatorial complexity of $\aA$. If $f$ only takes finite values, {and the domain of $\rel A$ countable,} then we say that $\aA$ is \emph{$\omega$-categorical}. We say that $\aA$ has \emph{less than doubly exponential orbit growth} if $\lim_{k\to\infty} {f(k)}/{2^{2^k}} = 0$.}
A~\emph{stabilizer} of a~group $\Aut(\aA)$ (resp.\ a clone $\Pol(\aA)$) is a~group of the form $\Aut(\aA,a_1,\dots,a_k)$ (resp.\ a~clone of the form $\Pol(\aA,a_1,\dots,a_k)$) where $a_1,\dots,a_k$ are elements from~$A$; here, $(\aA,a_1,\dots,a_k)$ denotes the expansion of $\aA$ by the unary relations $\{a_1\},\ldots,\{a_k\}$.

A first-order formula $\phi(x_1,\dots,x_n)$ is \emph{primitive positive} (\emph{pp}, for short) if it is of the form $\exists y_1,\dots,y_m\;\bigwedge_i R_i(\overline{z_i})$.
A relation $R\subseteq A^n$ is \emph{first-order definable} (resp.\ \emph{pp-definable}) in $\aA$ 
if there exists a~first-order formula $\phi$ (resp.\ pp-formula) such that $(a_1,\dots,a_n)\in R$ if and only if $\phi(a_1,\dots,a_n)$ holds in $\aA$, for all $a_1,\ldots,a_n\in A$.
A structure $\bB$ is a~\emph{first-order reduct} of $\aA$ if $\bB$ and $\aA$ have the same domain and if every relation of $\bB$ is first-order definable in $\aA$.

Uniformly continuous minion homomorphisms between clones have a~counterpart for relational structures which we define next.
Let $\rel {A}, \rel {B}$ be relational structures. We say that $\rel{B}$ is a~\emph{pp-power} of $\rel{A}$ if it is isomorphic to a~structure with domain $A^n$, where $n\geq 1$, whose relations are pp-definable from $\rel{A}$; here, a~$k$-ary relation on $A^n$ is regarded as a~$(k\cdot n)$-ary relation on $A$. We say that $\bB$ is \emph{pp-constructible} from $\rel{A}$ if it is homomorphically equivalent to a~pp-power of $\rel{A}$.
The following theorem ties together the notions of pp-constructibility and minion homomorphisms.
\begin{theorem}[Theorem 1.8 in~\cite{BOP18}]\label{thm:bop18}
	Let $\rel A$ be an $\omega$-categorical structure and let $\rel B$ be a~finite structure. Then $\rel B$ is pp-constructible from $\rel A$ if, and only if,
	there exists a~uniformly continuous minion homomorphism from $\Pol(\rel A)$ to $\Pol(\rel B)$.
\end{theorem}
We note, and are going to use, that the ``only if'' part of the statement above holds for arbitrary structures $\rel A$ and $\rel B$.

If $\aA$ and $\bB$ are $\sigma$-structures such that $B\subseteq A$ and such that for every $R\in\sigma$ of arity $k$, $R^{\rel A}\cap B^k = R^{\rel B}$, then we say
that $\bB$ is a~\emph{substructure} of $\aA$.
A structure $\aA$ is \emph{homogeneous} if for every two finite substructures $\rel B,\rel C$ and every isomorphism $f\colon\rel B\to\rel C$, there exists an automorphism $\alpha$ of~$\rel A$
such that $\alpha|_B = f$.
We note that if $\aA$ is homogeneous and its signature is finite, then $\aA$ has less than doubly exponential orbit growth.

A structure $\rel A$ {in a finite signature} is \emph{finitely bounded} if there exists a~finite set $\mathcal F$ of finite structures such that for every finite structure $\rel B$ with the same signature as $\rel A$,
$\rel B$ embeds into $\rel A$ if, and only if, no structure from $\mathcal F$ embeds into $\rel B$.
This is equivalent to saying that the class of finite substructures of $\rel A$ is definable by a~first-order universal sentence.

{The structures constructed in Theorem~\ref{thm:main-1} are finitely bounded but not necessarily homogeneous.
However, they are \emph{homogenizable} in the sense that they can be made homogeneous by adding finitely many relations.
In particular, they belong to the class of \emph{reducts of finitely bounded homogeneous structures}, which is the scope of the
infinite-domain tractability conjecture (Conjecture~\ref{conj:tract2}).}

An $\omega$-categorical structure $\rel A$ is a \emph{model-complete core} if for every endomorphism $e\colon\rel A\to\rel A$ and every finite subset $S$ of $A$, there exists an automorphism $\alpha\in\Aut(\rel A)$
such that $\alpha|_S = e|_S$. 

\begin{theorem}[\cite{cores-journal,BKOPP17a}]\label{thm:existence-mc-core}
Let $\rel A$ be $\omega$-categorical. There exists an $\omega$-categorical model-complete core\/ $\rel B$ that is homomorphically equivalent to $\rel A$.
Moreover, $\rel B$ is {$\omega$-categorical and} unique up to isomorphism.
\end{theorem}
The structure $\rel B$ in the theorem above is referred to as \emph{the model-complete core of $\rel A$}.

%% file: 3_sigma_G.tex
\section{Siggers-like conditions associated with graphs}
  \label{sec:sigma-g}

We show that for any non-trivial height 1 condition $\Sigma$, there is a non-trivial height 1 condition of a certain specific form, \changed{encoded by} a finite undirected graph, which is implied by $\Sigma$. Namely, from any finite undirected graph ${\rel G} = (V, E)$, one can construct a height 1 condition $\Sigma_{\rel G}$ in the following way: for each $v\in V$, one introduces a~ternary function  symbol $f_v$, and for each edge $(u,v)\in E$, one introduces a~$6$-ary symbol $g_{(u,v)}$, and adds to $\Sigma_{\rel G}$ the identities
\begin{align*}
  f_u(x,y,z) &\equals g_{(u,v)}( x,y, x,z, y,z ) \\
  f_v(x,y,z) &\equals g_{(u,v)}( y,x, z,x, z,y ).
\end{align*}
This corresponds to the condition $\Sigma(\rel K_3,\rel G)$ constructed in~\cite[Section 3.2]{BBKO19}. 
To give a~simple example, observe that if ${\rel G}$ consist of a~single vertex $v$ with an edge  $(v,v)$, then $\Sigma_{\rel G}$ is the Siggers condition (the function $g_{(v,v)}$ must satisfy the Siggers identity). We are now going to see that the Siggers condition is the strongest condition of this form; for clones over finite sets, it follows from~\cite{Sig10} that 
it is also a weakest among all non-trivial height~1 conditions, and thus \changed{for clones over finite sets} all non-trivial conditions of the form $\Sigma_{\rel G}$ are equivalent.

\begin{lemma} \label{lem:hom-between-sigma_g} \label{lem:3.1}
  Let ${\rel G}$ and ${\rel H}$ be finite graphs. If ${\rel G}$ maps homomorphically into ${\rel H}$, then $\Sigma_{\rel H}$ implies $\Sigma_{\rel G}$.
\end{lemma}

\begin{proof}
  Assume that $\Sigma_{\rel H}$ is satisfied in some clone $\clo C$, and fix functions $f_v\in\clo C$ for every vertex $v$ of $\hH$ and functions $g_e\in\clo C$ for every edge $e$ of $\hH$ witnessing this fact. Let $h\colon {\rel G}\to {\rel H}$ be a~homomorphism. For every vertex $v$ of $\gG$ we set $f'_v:=f_{h(v)}$, and for every edge $(u,v)$ of $\gG$ we set $g'_{(u,v)} = g_{(h(u),h(v))}$ (using the fact that $h$ is a~homomorphism). Then these functions witness the satisfaction of $\Sigma_{\rel G}$ in $\clo C$.
\end{proof}

\changed{We show next that} the condition $\Sigma_{\rel G}$ essentially forces the graph ${\rel G}$ into any graph which is compatible with $\Sigma_{\rel G}$ and which contains $\rel K_3$.

\begin{lemma} \label{lem:sigma_g}
  Let ${\rel G}$ be a~graph. Then ${\rel G}$ maps homomorphically to any graph ${\rel H}$ that contains $\rel K_3$, and whose polymorphisms satisfy $\Sigma_{\rel G}$.
\end{lemma}

\begin{proof}
   Let $v_1,v_2$ and $v_3$ be vertices of some copy of $\rel K_3$ in ${\rel H}$, and assume that we have polymorphisms of ${\rel H}$ satisfying the condition $\Sigma_{\rel G}$. Fix for every vertex $v$ of \changed{$\gG$} a function $f_v$ and for every edge $e$ of \changed{$\gG$} a function $g_e$ which witness this fact. We claim that the mapping $h\colon \gG\to\hH$ which sends every vertex $v$ of $\gG$ to $f_v(v_1,v_2,v_3)$ is a~homomorphism. Indeed, if $(u,v)$ is an edge of ${\rel G}$ then we get
  \begin{align*}
    f_u(v_1,v_2,v_3) &= g_{(u,v)}(v_1,v_2,v_1,v_3,v_2,v_3) \\
    f_v(v_1,v_2,v_3) &= g_{(u,v)}(v_2,v_1,v_3,v_1,v_3,v_2)\; .
  \end{align*}
Since $g_{(u,v)}$ is a polymorphism of $\hH$, and since $(v_i,v_j)$ is an edge in $\hH$ for all $i\neq j$, we get that $g_{(u,v)}(v_1,v_2,v_1,v_3,v_2,v_3)$ and $g_{(u,v)}(v_2,v_1,v_3,v_1,v_3,v_2)$ are related by an edge in $\hH$. Hence, $(h(u),h(v)) = (f_u(v_1,v_2,v_3),f_v(v_1,v_2,v_3))$ is an edge of~${\rel H}$.
\end{proof}

Finally, these tools allow us to provide a simple criterion for the triviality of conditions of the form $\Sigma_{\rel G}$. Even though the following lemma follows directly from \cite[Lemma 3.14]{BBKO19}, we include a~proof for completeness.

\begin{lemma}(cf.\ \cite[Lemma 3.14]{BBKO19}) \label{lem:1.1}
  For any finite graph ${\rel G}$, the condition $\Sigma_{\rel G}$ is trivial if and only if ${\rel G}$ is $3$-colorable.
\end{lemma}

\begin{proof}
  First, assume that ${\rel G}$ is $3$-colorable, i.e., it possesses a~homomorphism to $\rel K_3$. Then by the previous lemma, we have that $\Sigma_{\rel G}$ is implied by $\Sigma_{\rel K_3}$, and therefore it is enough to show that $\Sigma_{\rel K_3}$ is trivial. That is, we have to assign projections to the symbols of $\Sigma_{\rel K_3}$ in such a way that the identities are satisfied.
     Let $1,2,3$ be the vertices of $\rel K_3$, and define $f_i$ to be the $i$-th ternary projection. Moreover, for $i\neq j$ assign to $g_{(i,j)}$ the unique $6$-ary projection so that 
     \begin{align*}
  f_i(x,y,z) &\equals g_{(i,j)}( x,y, x,z, y,z ) \\
  f_j(x,y,z) &\equals g_{(i,j)}( y,x, z,x, z,y ).
\end{align*}
are satisfied.  By definition, this assignment satisfies $\Sigma_{\rel K_3}$. 

  If ${\rel G}$ is not $3$-colorable, then Lemma~\ref{lem:sigma_g} implies that $\Pol(\rel K_3)$ does not satisfy $\Sigma_{\rel G}$, and hence $\Sigma_{\rel G}$ is non trivial.
\end{proof}

\begin{remark}
  The lemma implies that the problem of deciding the triviality of height 1 conditions is \NP-hard, since it provides a reduction from the 3-coloring problem. The problem of deciding whether a~given height 1 condition is trivial is  known (in a~different, but equivalent  formulation) in computer science under the name \emph{Label Cover}~\cite{ABSS97}.
\end{remark}

We now show that for each non-trivial height 1 condition $\Sigma$ there is a~non-3-colorable graph ${\rel G}$ such that $\Sigma_{\rel G}$ is implied by $\Sigma$ (cf.\ \cite[Theorem 4.12]{BBKO19}).
We will use the folklore fact that $\Pol(\rel K_3)$ does not satisfy any non-trivial height 1 condition since it only contains functions of the form
\(
  f(x_1,\dots,x_n) = \alpha(x_i)
\)
where $1\leq i\leq n$ and $\alpha\colon \rel K_3 \to \rel K_3$ is a bijection. In particular, there exists a~minion homomorphism from $\Pol(\rel K_3)$ to $\Proj$.

\begin{lemma} \label{lem:1.2}
  Let $\clo A$ be a~clone that does not have a~minion homomorphism to $\Proj$. Then there exists a non $3$-colorable finite graph ${\rel G}$  such that $\clo A$ satisfies $\Sigma_{\rel G}$.
\end{lemma}

\begin{proof}
By~\cite[Lemma 4.4]{BBKO19}, minion homomorphisms from $\clo A$ to $\Pol(\rel K_3)$ correspond precisely to 3-colorings of a certain graph ${\rel F}=(V,E)$, which we shall now describe (cf.\ \cite[Definition 4.1]{BBKO19}). This (possibly infinite) graph will serve as a~source of finite graphs ${\rel G}$ such that $\clo A$ satisfies $\Sigma_{\rel G}$.

We take $V:= \clo A^{(3)}$, and define the edges of ${\rel F}$ in the following way: $(f_1,f_2)\in E$ if and only if there exists $g\in \clo A^{(6)}$ such that
  \begin{align*}
    f_1(x,y,z) &\equals g( x,y,x,z,y,z ) \\
    f_2(x,y,z) &\equals g( y,x,z,x,z,y )
  \end{align*}
  holds in $\clo A$. 
  Clearly, $\clo A$ satisfies $\Sigma_{\rel G}$ for each finite subgraph ${\rel G}$ of ${\rel F}$ since the functions that correspond to the vertices of ${\rel G}$ together with the witnesses for the edges of ${\rel G}$ provide a~solution to $\Sigma_{\rel G}$ (see also~\cite[Lemma 4.3]{BBKO19}).
  
  Now \cite[Lemma 4.4]{BBKO19} (applied to $\mathbf A=\mathbf B:=\rel K_3$) states that the minion homomorphisms to from $\clo A$ to $\Pol(\rel K_3)$ correspond precisely to the 3-colorings of $\rel F$. Since $\clo A$ does not have any minion homomorphism to $\Proj$, it has none to $\Pol(\rel K_3)$ either, and hence $\rel F$ is not $3$-colorable. By a~standard compactness argument, there exists a finite subgraph $\gG$ of $\rel F$ which is not 3-colorable. Since $\clo A$ satisfies $\Sigma_{\rel G}$, the proof is complete.
\end{proof}

\begin{corollary} \label{cor:1.2}
  For each non-trivial height 1 condition $\Sigma$ there exists a~graph ${\rel G}$ that is not $3$-colorable and such that $\Sigma_{\rel G}$ is implied by $\Sigma$.
\end{corollary}

\begin{proof}
  Let $\clo A$ be the clone of term operations of the free countably generated algebra in the variety defined by $\Sigma$. Clearly, $\Sigma$ witnesses that $\clo A$ has no minion homomorphism to ${\Proj}$. Therefore, Lemma~\ref{lem:1.2} provides a~non $3$-colorable graph ${\rel G}$ such that $\clo A$ satisfies $\Sigma_{\rel G}$.
  Since $\clo A$ is free, we obtain that $\Sigma_{\rel G}$ is implied by $\Sigma$.
\end{proof}

Corollary~\ref{cor:1.2} states that the family $\{\Sigma_{\rel G} \mid \rel G \text{ is a non-3-colorable finite graph}\}$ is a weakest family of non-trivial height~1 conditions in the sense that any nontrivial height~1 condition implies one of its members. In the rest of this section, we will turn this family into a decreasing chain with the same property, obtaining a Mal'cev condition which characterizes non-triviality.
We remark that such a chain of height 1 conditions that are not of our special form, i.e., a chain $(\Sigma_n)_{n\geq 1}$ of non-trivial height~1 conditions such that 
\begin{itemize}
  \item $\Sigma_n$ implies $\Sigma_{n+1}$ for all $n\geq 1$, and
  \item for every non-trivial height~1 condition $\Sigma$ there exists $n\geq 1$ such that $\Sigma$ implies $\Sigma_n$,
\end{itemize}
can be constructed easily. The key is to observe that, for any two non-trivial height~1 conditions $\Sigma$ and $\Pi$, there is a non-trivial height 1 condition $\Delta$ which is implied by both of them. The following {natural construction has been pointed out to us by} Pierre Gilibert:
the symbols of $\Delta$ are all pairs $(f,g)$, where $f$ is a symbol of $\Sigma$ and $g$ is a symbol of $\Pi$. The arity of $(f,g)$ is the sum of the arities of $f$ and $g$.  The identities are all identities of the form
\[
  (f,g)(x_1,\dots,x_n,z_1,\dots,z_m) \approx  (f',g)(y_1,\dots,y_l,z_1,\dots,z_m)
\]
where $f(x_1,\dots,x_n)\approx f'(y_1,\dots,y_l)$ is an identity of $\Sigma$, or, similarly, of the form
\[
  (f,g)(x_1,\dots,x_n,y_1,\dots,y_l) \approx  (f,g')(x_1,\dots,x_n,z_1,\dots,z_m)
\]
where $g(y_1,\dots,y_l)\approx g'(z_1,\dots,z_m)$ is an identity of $\Pi$. It is easy to check that $\Delta$ is implied by either of $\Sigma$ and $\Pi$, and also that $\Delta$ is non-trivial given that $\Sigma$ and $\Pi$ are.
{The existence of the chain $(\Sigma_n)_{n\geq 1}$ clearly follows: starting with an enumeration $\Pi_1,\Pi_2,\dots$ of all non-trivial height 1 conditions, one defines $\Sigma_1 = \Pi_1$, and $\Sigma_{n+1}$ to be a non-trivial height 1 condition that is implied by both $\Sigma_n$ and $\Pi_{n+1}$ for all $n\geq 1$.}

Now, a chain which is composed of conditions of the form $\Sigma_{\rel G}$ can be constructed by interweaving the above inductive construction with repeated use of Corollary~\ref{cor:1.2}. This procedure gives us, however, no control over the graphs which appear in the chain; moreover our proof of Corollary~\ref{cor:1.2} is non-constructive as it involves a compactness argument, although this might be remedied with a bit of care. We give an alternative direct construction which uses the fact that the product of two non 3-colorable graphs is not 3-colorable \cite{ElZaharSauer}\footnote{This is a special case of Hadetniemi's conjecture; in full generality, the Hedetniemi's conjecture has recently been disproved~\cite{Hedetniemi-false}.}.
One more alternative construction, that does not rely on this deep result, can be found in Appendix~\ref{sect:a1}. Both mentioned constructions allow the chain to be naturally enumerated by a Turing machine.

\begin{lemma} \label{lem:chain}
  There exist {finite} graphs $\rel H_1$, $\rel H_2$, \dots\ such that
  \begin{enumerate}
    \item $\rel H_n$ is not 3-colourable for any $n\geq 1$,
    \item $\rel H_{n+1}$ maps homomorphically to $\rel H_n$ for all $n\geq 1$, and
    \item for every non-trivial height~1 condition $\Sigma$ there exists $n\geq 1$ such that $\Sigma$ implies $\Sigma_{\rel H_n}$.
  \end{enumerate}
\end{lemma}

\begin{proof}
  Fix an enumeration $\rel G_1,\rel G_2,\dots$ of all non 3-colorable graphs. For {$n\geq 1$, define} $\rel H_n := \rel G_1 \times \dots \times \rel G_n$.
  By~\cite{ElZaharSauer} we get that the graphs $\rel H_n$ are not 3-colorable, which gives item (1). Item (2) follows {since  $\rel H_{n+1}$ projects into $\rel H_{n}$ for all $n\geq 1$}. {To prove item (3), assume that $\Sigma$ is a~non-trivial height~1 condition. Then there exists $n \geq 1$  such that $\Sigma_{\rel G_n}$ is weaker than $\Sigma$, by Corollary~\ref{cor:1.2}. Since $\rel H_n$ maps homomorphically to $\rel G_n$ via a projection, we get that $\Sigma_{\rel H_n}$ is implied by $\Sigma$ as requested.}
\end{proof}

{In the next section, we are going to prove Theorem~\ref{thm:main-1} which states that there exists no weakest non-trivial height~1 condition. A consequence of this theorem is that we can make our chain strictly decreasing.}

\begin{theorem} \label{thm:chain}
  There exist {finite} graphs $\rel H_1$, $\rel H_2$, \dots\ such that
  \begin{enumerate}
    \item $\Sigma_{\rel H_n}$ is a~non-trivial height 1 condition for all $n\geq 1$,
    \item $\Sigma_{\rel H_n}$ is strictly stronger than $\Sigma_{\rel H_{n+1}}$ for all $n\geq 1$, and
    \item for every non-trivial height~1 condition $\Sigma$ there exists $n\geq 1$ such that $\Sigma$ implies $\Sigma_{\rel H_n}$.
  \end{enumerate}
\end{theorem}

\begin{proof}[Proof given Theorem~\ref{thm:main-1}]
  Assume that $\rel H_1, \rel H_2,\ldots$ are such that they satisfy the claim of Lemma~\ref{lem:chain}. In particular, we have that ${\rel H_n}$ is not 3-colorable for any $n\geq 1$ and hence $\Sigma_{\rel H_n}$ is non-trivial (Lemma~\ref{lem:1.1}). We also have that $\rel H_{n+1}$ maps homomorphically to $\rel H_n$ which, together with Lemma~\ref{lem:3.1}, gives that $\Sigma_{\rel H_n}$ implies $\Sigma_{\rel H_{n+1}}$ for all $n\geq 1$. It thus follows from Theorem~\ref{thm:main-1} that the sequence $(\Sigma_{\rel H_n})_{n\geq 1}$  must strictly decrease infinitely often. Taking the graphs corresponding to a~strictly decreasing subsequence proves the theorem.
\end{proof}

%% file: 5_no_strong.tex
\section{There is no weakest height~1 condition}\label{sect:noweakest}

We now show that there is no weakest height~1 condition, even when restricted to \changed{structures which are} reducts of finitely bounded homogeneous structures in a~finite relational language.  More precisely, for each non-trivial height~1 condition $\Sigma$ there is a~structure $\aA$ such that
\begin{itemize}
\item $\aA$ has a~finite relational language, and is a~reduct of a~finitely bounded homogeneous structure;
\item $\CSP(\aA)$ is in \Ptime;
\item $\Pol(\aA)$ satisfies some non-trivial height~1 condition (or equivalently, does not possess a~minion homomorphism to~$\Proj$);
\item $\Pol(\aA)$ does not satisfy $\Sigma$.
\end{itemize}
It follows that in the dichotomy conjecture for reducts of finitely bounded homogeneous structures, which currently characterizes tractability of a~CSP by the existence of a~pseudo-Siggers polymorphism, the latter cannot be replaced by any height 1 condition.
This is in contrast with the CSP dichotomy for finite structures (Theorem~\ref{thm:bulatov-zhuk}), which does draw the borderline between tractability and hardness by such a~condition, for example, {the existence of} a~Siggers polymorphism.

Our structures will be obtained as universal structures for graphs with forbidden homomorphic images, first constructed by Cherlin, Shelah, and Shi~\cite{CSS99} and later refined by Hubi\v{c}ka and Ne\v{s}et\v{r}il~\cite{HN16}.

\begin{definition}
For a family of $\sigma$-structures $\mathcal G$, we set $\Forb(\mathcal G)$ to be the class of all $\sigma$-structures which do not contain a~homomorphic image of any member of $\mathcal G$.
A countable structure is \emph{universal} for  $\Forb(\mathcal G)$ if it embeds precisely those countable structures which are elements of $\Forb(\mathcal G)$.
\end{definition}

In the following, a~\emph{cut} of a~relational structure $\rel G$ is defined to be a~set of elements of $\rel G$ whose removal disconnects the Gaifman graph of $\rel G$ (the graph with same domain as $\rel G$ and where there is an edge $\{x,y\}$ iff $x$ and $y$ appear together in some tuple of some relation of $\rel G$). The structure $\rel G$ is \emph{connected} if its Gaifman graph is.
{An $\omega$-categorical} structure $\rel A$ has \emph{no algebraicity} if for all $k\geq 0$ and $a_1,\dots,a_k\in A$, the finite orbits of the stabilizer $\Aut(\rel A,a_1,\dots,a_k)$
are {precisely the sets} $\{a_1\},\dots,\{a_k\}$.

\begin{theorem}[{Corollary of \cite[Theorem 3.3]{HN16}}]\label{thm:hubicka-nesetril}
  Let $\mathcal G$ be a finite family of finite connected structures. There exists a~countable $\omega$-categorical structure $\CSS({\mathcal G})$ with the following properties:
  \begin{itemize}
    \item $\CSS({\mathcal G})$ is universal for $\Forb(\mathcal G)$,
    \item $\CSS({\mathcal G})$ has no algebraicity,
    \item there exists a~homogeneous expansion of\/ $\CSS({\mathcal G})$ by finitely many pp-definable relations whose arities are the size of the minimal cuts of the structures in $\mathcal G$.
    Moreover, this expansion is finitely bounded.
  \end{itemize}
\end{theorem}
{Note that a~consequence of the third item in Theorem~\ref{thm:hubicka-nesetril} is that $\CSS(\mathcal G)$ resides within the scope of the infinite-domain CSP dichotomy conjecture. 
We simply write $\CSS(\gG)$ when $\mathcal G$ consists of a single structure $\gG$.}

\begin{definition}\label{def:qnu}
  We say that a~function $f\colon A^n \to A$ is a~\emph{quasi near unanimity operation} if it satisfies the identities 
  \[
    f(y,x,\dots,x) \equals f(x,y,x,\dots,x) \equals \dots
      \equals f(x,\dots,x,y) \equals f(x,\dots,x) \; ,
    \label{eq:qnu}
  \]
  i.e., if it takes the same value on all tuples that consist of a~single value $x\in A$ with at most one exception.
\end{definition}
\changed{Quasi near unanimity operations which in addition are \emph{idempotent}, i.e., which additionally satisfy $f(x,\dots,x)\equals x$, are called \emph{near unanimity operations} in the literature, and have been widely studied (see, for example,~\cite{BP}).}
Also note that the identities in Definition~\ref{def:qnu} constitute a~non-trivial height~1 condition.

\begin{lemma}\label{lem:fg}
  Let $\gG$ be a~finite connected graph which is not 3-colorable, and let $\hH$ be universal for $\Forb(\gG)$. Then:
\begin{itemize}
\item $\Pol(\hH)$ does not satisfy $\Sigma_{\gG}$;
\item $\Pol(\hH)$ has quasi near unanimity polymorphisms of all arities larger than the number of edges of $\gG$.
\end{itemize}  
\end{lemma}

\begin{proof}
First, we prove that $\Pol(\hH)$ does not satisfy $\Sigma_{\gG}$. Observe that $\hH$ contains an isomorphic copy of $\rel K_3$; on the other hand, it does not contain a~homomorphic image of $\gG$. The latter is clear from the definition, and the former follows from the assumption that $\gG$ is not $3$-colorable, which implies that there is no homomorphism from $\gG$ to $\rel K_3$, hence $\rel K_3$ embeds into $\hH$ by universality. The claim then follows from Lemma~\ref{lem:sigma_g}.

For the second claim, let $n$ be larger than the number of edges of $\gG$.  To show that $\Pol(\hH)$ contains a~quasi near unanimity operation of arity $n$, we use the \emph{indicator structure} for this condition. It is obtained by factoring the $n$-th Cartesian power $\hH^n$ of $\hH$ by the equivalence relation $\sim$ which identifies all sets of tuples of the form 
$$
\{(x,\dots,x,y),\dots,(y,x,\dots,x), (x,\ldots,x)\}\; .
$$ 
There is an edge in $\hH^n{/}{\sim}$ between two equivalence classes $A$ and $B$ if and only if there exist $(u_1,\dots,u_n)\in A$ and $(v_1,\dots,v_n)\in B$ such that $(u_i,v_i)$ is an edge in $\hH$ for all $1\leq i\leq n$.
We now argue that the graph $\hH^n{/}{\sim}$ thus obtained is an element of $\Forb(\gG)$, since if that is the case, then it embeds into $\hH$ by universality. This embedding provides the requested quasi near unanimity polymorphism of $\hH$ by composing it with the factor map from $\hH^n$ to $\hH^n{/}{\sim}$.

Assume for contradiction that there exists a~homomorphism $h\colon \gG\to \hH^n{/}{\sim}$. Let us call $n$-tuples which are constant except for at most one value \emph{almost constant}. These are precisely the tuples whose equivalence class with respect to $\sim$ consists of more than one element, or equivalently, contains a~constant tuple. When $u$ is a~vertex of $\gG$, then we write $(u_1,\ldots,u_n)$ for the representative of the equivalence class $h(u)$ which is constant, when $h(u)$ contains such a~representative, and which is the only representative of its class otherwise. Observe that if $(u,v)$ is an edge of $\gG$ and
  \begin{itemize}
    \item $h(u)$ and $h(v)$ are both almost constant, then $(u_i,v_i)$ is an edge of $\hH$ for all $1\leq i\leq n$; the same applies if $h(u)$ and $h(v)$ are both \emph{not} almost constant;
    \item otherwise, $(u_i,v_i)$ is an edge of $\hH$ except for possibly one index $i\in\{1,\dots,n\}$.
  \end{itemize}
Therefore, \changed{by the choice of $n$,} there exists $1\leq i\leq n$ such that \changed{for every edge $(u,v)$ of $\gG$, $(u_i,v_i)$ is an edge of $\hH$}. But then the mapping which sends every $u\in\gG$ to $u_i$ is a~homomorphism from $\gG$ into $\hH$, a~contradiction.
\end{proof}

\begin{lemma}\label{lem:forbCSP}
  Let $\gG$ be a~finite graph, and let $\hH$ be universal for $\Forb(\gG)$. Then $\CSP(\hH)$ is solvable in polynomial time.
\end{lemma}
\begin{proof}
This is obvious, as $\CSP(\hH)$ corresponds to the problem of determining whether there exists a~homomorphism from $\gG$ (which is fixed)
to an input graph $\hH'$, and there are at most $|\hH'|^{|\gG|}$ such homomorphisms.
\end{proof}

We finally prove Theorem~\ref{thm:main-1} stated in the introduction.

\begin{proof}
Let $\Sigma$ be a~non-trivial height 1 condition. By Corollary~\ref{cor:1.2} there exists a~non 3-colorable graph $\gG$ such that $\Sigma_\gG$ is weaker than $\Sigma$.  If $\gG$ is not connected, then one of its connected components $\cC$ is {not} 3-colorable.  The height~1 condition $\Sigma_{\cC}$ is non-trivial by Lemma~\ref{lem:1.1}, and clearly weaker than $\Sigma_\gG$. By Lemma~\ref{lem:fg}, $\Pol(\CSS(\cC))$ does not satisfy $\Sigma_\cC$ but has a~quasi near unanimity operation of sufficiently large arity.
In particular, $\Pol(\CSS(\cC))$ does not satisfy $\Sigma$, but satisfies some non-trivial height 1 condition.
\end{proof}

%% file: 6_topology.tex
\section{Topology is relevant}
  \label{sec:topology}

The original CSP dichotomy conjecture for reducts of finitely bounded homogeneous structures due to Bodirsky and Pinsker (see~\cite{BPP-projective-homomorphisms}) \changed{envisions} the following borderline between tractability and hardness.

\begin{conjecture}\label{conj:tract}
Let $\aA$ be a~reduct of a~finitely bounded homogeneous structure. Exactly one of the following holds:
\begin{enumerate}
  \item some stabilizer of the polymorphism clone of its model-complete core possesses a~continuous clone homomorphism to $\Proj$, and $\CSP(\aA)$ is \NP-complete,
  \item no stabilizer of the polymorphism clone of its model-complete core possesses a~continuous clone homomorphism to $\Proj$, and $\CSP(\aA)$ is in \Ptime.
\end{enumerate}
\end{conjecture}

Barto and Pinsker showed in~\cite{BartoPinskerDichotomy, Topo} that topology is irrelevant in this conjectured borderline, since the word `continuous' can simply be dropped without changing the \changed{force of} the conjecture. More precisely, when $\aA$ is any $\omega$-categorical structure with model-complete core $\bB$, then some stabilizer of $\Pol(\bB)$ possesses a~clone homomorphism to $\Proj$ if and only if some stabilizer of $\Pol(\bB)$ possesses a~continuous such homomorphism, and this is witnessed by the non-satisfaction of the pseudo-Siggers identity in $\Pol(\bB)$.

Following the discovery of the importance of minion homomorphisms for the complexity of CSPs in~\cite{BOP18}, it was then shown that whenever $\aA$ is any $\omega$-categorical structure with less than doubly exponential orbit growth (a condition satisfied in particular by all structures in the range of the conjecture), then the above hardness criterion is equivalent to the existence of a~uniformly continuous minor preserving map from $\Pol(\aA)$ to $\Proj$~\cite{BKOPP17, BKOPP17a}.

Naturally, the question of whether topology is irrelevant also for minion homomorphisms was raised in this context~\cite{BartoPinskerDichotomy, BOP18, Topo}, in particular for $\omega$-categorical structures with less than doubly exponential orbit growth.

\begin{question} \label{quest:topo}
Let $\aA$ be an $\omega$-categorical structure.
If there exists a~minion homomorphism $\Pol(\aA)\to\Proj$, does there exist a~uniformly continuous one?
\end{question}

While a~positive answer was obtained in some special cases~\cite{BKOPP17, BKOPP17a}, we are going to provide a~negative answer to the question in general.
The remainder of this section will be devoted to the construction of the structure $\mathbb S$ and the verification of the properties claimed in Theorem~\ref{thm:main-3}.

\subsection{Encoding graphs in higher arities}

Our first step will be a~standard construction which allows us to encode graphs as structures on $n$-tuples, for arbitrary $n\geq 1$.

\begin{lemma}\label{lem:S}
  Let $\gG$ be a~finite connected loopless graph and $n \geq 1$. Then there exists a~structure $\sS(\gG,n)$ with a~single relation $R$ of arity $2n$ such that
  \begin{enumerate}
    \item The expansion $(\sS(\gG,n),\neq)$ of $\sS(\gG,n)$ by the inequality relation is an $\omega$-categorical model-complete core without algebraicity;
    \item $\sS(\gG,n)$ pp-constructs the Cherlin-Shelah-Shi structure $\CSS(\gG)$;
    \item The relation $R$ of\/ $\sS(\gG,n)$ only contains tuples with pairwise distinct entries;
    \item $\Aut(\sS(\gG,n))$ has at most $3^{k^{n|\gG|}}$ orbits of $k$-tuples, for every $k \geq 2$.
  \end{enumerate}
\end{lemma}

\begin{proof}
  The structure $\sS(\gG,n)$ is itself obtained via a~$\CSS$ structure for a finite family $\mathcal G$ of structures. Let $\gG'$ be the structure obtained from $\gG$ by replacing each vertex $x$ of $\gG$ by an $n$-tuple $\overline x$ of new distinct elements, and requiring, for vertices $x,y$ of $\gG$, the $2n$-ary relation $R(\overline x, \overline y)$ to hold if $(x,y)$ is an edge in $\gG$.  Note that $\gG'$ is connected because $\gG$ is connected.

Let $\mathcal G$ contain $\gG'$ as well as every connected structure on $<2n$ elements and containing a single $R$-tuple (such a structure is called \emph{loop-like} in the following).
Let $\mathbb{F}'$ be the $\CSS$ structure for $\mathcal G$.
Then $\mathbb{F}'$ is $\omega$-categorical and has no algebraicity. We set $(\sS(\gG,n),\neq)$ to be the model-complete core of the structure $(\mathbb{F}',\neq)$; it is also $\omega$-categorical~\cite{cores-journal}, and it follows from its construction that it has no algebraicity either (cf.\ the proof of Theorem~27 in~\cite{MMSNP}). Hence, item~(1) is satisfied.

We now show item~(2).  Let $\mathbb{T}$ be the graph whose vertices are the $n$-tuples of $\sS(\gG,n)$, and where two tuples $(x_1,\dots,x_n)$ and $(y_1,\dots,y_n)$ are related if and only if $R(x_1,\dots,x_n,y_1,\dots,y_n)$ holds in $\sS(\gG,n)$. Clearly, $\mathbb{T}$ is a~pp-power of $\sS(\gG,n)$. We claim that 
  $\mathbb{T}$ and $\CSS(\gG)$ are homomorphically equivalent; this implies that $\CSS(\gG)$ is pp-constructible from $\sS(\gG,n)$, as required.
  
  To prove the claim, by a~standard compactness argument it is sufficient to show that a~finite graph homomorphically maps into $\mathbb{T}$ if and only if it homomorphically maps into $\CSS(\gG)$; in other words, a~finite graph homomorphically maps into $\mathbb{T}$ if and only if it does not contain a~homomorphic image of $\gG$. Suppose first that there existed a~homomorphism from $\gG$ into $\rel T$. Then $\gG'$ constructed as above would have a~homomorphism into $\sS(\gG,n)$, and hence also into $\rel F'$, a~contradiction.  Conversely, if $\hH$ is a~graph which does not contain a~homomorphic image of $\gG$, then $\gG'$ does not homomorphically map into $\hH'$, and therefore $\hH'$ embeds into $\rel F'$, and hence homomorphically maps into $\sS(\gG,n)$. But this implies that $\hH$ homomorphically maps to $\rel T$.

  Item $(3)$ of the lemma holds since we have included loop-like obstructions in the definition of $\rel F'$, and since $\rel F'$ and $\sS(\gG,n)$ are homomorphically equivalent. 
  
  To see item~(4), note that the orbit-growth of a~homogeneous structure with relations of arity at most $r$ is bounded by $3^{k^r}$ for large enough $k$.
  By Theorem~\ref{thm:hubicka-nesetril}, $\CSS(\gG')$ has a~homogeneous expansion by relations with arity at most $|\gG'|=n|\gG|$.
  Thus, $\Aut(\mathbb{F}')$ has for large $k$ at most $3^{k^{n|\gG|}}$ orbits of $k$-tuples.
  Whence, the same holds for the model-complete core $\sS(\gG,n)$, which has at most the number of orbits of the original structure {(see~\cite{Bodirsky-HDR}, Proposition 3.6.24)}. 
  \end{proof}

\subsection{Superposition of the encodings}
It is well-known that if two $\omega$-categorical structures $\mathbb A$ and $\mathbb B$ in disjoint signatures $\sigma$ and $\tau$ have no algebraicity,
then there exists a~generic superposition $\aA\odot \bB$ of the two in the signature $\sigma\cup\tau$ and which is unique up to isomorphism.
This generic superposition is again $\omega$-categorical and without algebraicity.
It is obtained as follows:
\begin{enumerate}
	\item Expand $\aA$ by all relations that have a first-order definition in $\aA$, and similarly for $\bB$. Call $\aA'$ and $\bB'$ the resulting structures and let $\sigma'$ and $\tau'$ be their signatures (that
	we take to be disjoint without loss of generality).
	\item Since $\aA$ and $\bB$ are without algebraicity, so are $\aA'$ and $\bB'$ (expanding by first-order definable relations does not change the automorphism groups of the structures).
	Thus, the class of finite substructures of $\aA'$ and $\bB'$ have the strong amalgamation property (see Proposition 2.15 in~\cite{Oligo}).
	\item The class of finite $(\sigma'\cup\tau')$-structures whose $\sigma'$- and $\tau'$-reducts embed into $\aA'$ and $\bB'$, respectively, has the strong amalgamation property,
	and we call $\aA'\odot \bB'$ its Fra\"iss\'e limit.
	The $(\sigma\cup\tau)$-reduct of $\aA'\odot\bB'$ is then our structure $\aA\odot\bB$.
\end{enumerate}
As an example, take $\aA$ to be $(\mathbb Q,<)$ and $\bB$ to be the random graph (i.e., the graph $\CSS(\mathbb L)$ where $\mathbb L$ is a graph on a single vertex with a loop).
Then $\aA\odot \bB$ is the \emph{random ordered graph}, i.e., the Fra\"iss\'e limit of the class of finite simple graphs with a~total ordering on the vertices.

The same construction works for generic superpositions of infinitely many $\omega$-categorical structures without algebraicity.
The generic superposition will have an infinite signature, but will be $\omega$-categorical if the Fra\"{i}ss\'{e} class which yields the superposition has finitely many inequivalent atomic formulas of each arity.

In our construction of the structure of Theorem~\ref{thm:main-3}, we would like to superpose the graphs {from Theorem~\ref{thm:chain}}; this superposition would however not be $\omega$-categorical as there would be infinitely many orbits of pairs of vertices. This is why we superpose encodings of these graphs on tuples of increasing arity instead.

\begin{construction}\label{constr:s}
Let $\hH_1,\hH_2,\dots$ be an enumeration of  
{all non-3-colorable finite graphs (in fact, the graphs from the chain in Theorem~\ref{thm:chain} are sufficient).} 
Let $g\colon\mathbb N\to\mathbb N$ be a \changed{strictly} increasing function. {We define $\sSa$ as the generic superposition of all of the structures $\sS(\hH_{n},g(n))$, for $n \geq 1$:}
\[
  \sSa:={\bigodot}_{n\geq 1} \sS(\hH_{n},g(n))\;.
\]
{Note that $\sSa$ depends on the enumeration of the graphs above as well as on the function $g$. The former dependence will be irrelevant for our purposes, but the latter will play a role, and we make the convention that in all statements we state all properties required of $g$ in order for them to hold.}
\end{construction}

We note that by Theorem~\ref{thm:hubicka-nesetril}, each $\sS(\hH_n,g(n))$ has an expansion by finitely many relations which is homogeneous.
The structure obtained by expanding $\sSa$ by this infinite set of relations is itself homogeneous.
In the proof below, we call this expansion `the' homogenization of $\sSa$, even though it is not unique.

\begin{lemma}\label{lem:core}
The structure $(\sSa,\neq)$ is an $\omega$-categorical model-complete core without algebraicity.
\end{lemma}
\begin{proof}
The generic superposition of structures without algebraicity {never has} algebraicity,
and expanding a~structure by $\neq$ does not introduce algebraicity since $\sSa$ and $(\sSa,\neq)$ have the same orbits.

We prove that $\sSa$ is $\omega$-categorical (which implies that $(\sSa,\neq)$ is $\omega$-categorical, by the sentence above).
First, we prove that every atomic formula $\phi(x_1,\dots,x_r)$ over $\sS(\hH_n,g(n))$
is either equivalent to ``false'' or has at least $g(n)$ different variables.
Suppose that $\phi$ is the relation symbol $R_n$ (and thus $r=2g(n)$). By construction of $\sS(\hH_n,g(n))$, since all loop-like structures have been forbidden,
we have that either all the variables are distinct, or $\phi(x_1,\dots,x_r)$ is not satisfiable in $\sS(\hH_n,g(n))$ and is equivalent to false.
Suppose now that $\phi$ is a~relation symbol added for the homogenization of $\sS(\hH_n,g(n))$.
Let $\hH_n'$ be the structure obtained from $\hH_n$ as in the proof of Lemma~\ref{lem:S}.
Note that the cuts of $\hH'_n$ have size at least $g(n)$, so we know from Theorem~\ref{thm:hubicka-nesetril} that $r\geq g(n)$.
Moreover, Theorem~\ref{thm:hubicka-nesetril} gives that $\phi$ is equivalent to a~pp-formula over $\sS(\hH_n,g(n))$. Then at least $g(n)$ of the variables of $\phi$ are different,
for otherwise a~clause in $\phi$ would be of the form $R_n(y_1,\dots,y_{2g(n)})$ with fewer than $2g(n)$ distinct variables, and $\phi$ would be equivalent to ``false''.
In conclusion, we obtain that all non-trivial atomic formulas over the homogenization of $\sS(\hH_n,g(n))$ have arity at least $g(n)$, and there are only finitely many of them since this homogenization
has a~finite signature.
Thus, since $g$ is an increasing function, the homogenization of $\sSa$ has only finitely many atomic formulas of each arity.
It follows that the homogenization of $\sSa$ has finitely many orbits of each arity, so that this homogenization is $\omega$-categorical, and thus $\sSa$ itself is $\omega$-categorical.

To see that $(\sSa,\neq)$ is a~model-complete core, let $e$ be an endomorphism of $(\sSa,\neq)$, and let $F$ be a~finite subset of its domain. Then $e$ is also, in particular, an endomorphism of $(\sS(\hH_n,g(n)),\neq)$ for all $n\geq 1$, and since the latter structures are model-complete cores, the restriction \changed{$e|_F$} of $e$ to $F$ has an expansion to an automorphism of $\sS(\hH_n,g(n))$ for each $n\geq 1$. It then follows that \changed{$e|_F$} is a~partial isomorphism of the Fra\"{i}ss\'{e} structure $\sS'$ of which $\sSa$ is the reduct. By homogeneity, \changed{$e|_F$} extends to an automorphism of $\sSa'$, which is also an automorphism of $\sSa$ and of $(\sSa,\neq)$.
\end{proof}

We show in the next lemma that the orbit growth of $\sSa$ can be controlled by picking a~suitable \changed{function} $g$ in the construction.

\begin{lemma}\label{lem:orbitgrowth}
For every increasing $f\colon\mathbb N\to\mathbb N$ that dominates every polynomial,
there exists a function $g \colon\mathbb N\to\mathbb N$ such that the number of orbits of $k$-tuples of\/ $\sSa$ is not asymptotically larger than \[3^{f(k)}.\]
In particular, there exists a function $g \colon {\mathbb N} \to {\mathbb N}$ such that $\sSa$ has less than doubly exponential orbit growth.
\end{lemma}
\begin{proof}
We construct $g$ by induction, first setting $g(1)=1$.
Suppose now that $g(1),\dots,g(n)$ are defined.
Since $f$ dominates every polynomial, there exists a~$k_n>g(n)$ such that $\sum_{i=1}^{n} k_n^{g(i)|\hH_i|} < f(k_n)$ for all $k\geq k_n$.
Let $g(n+1):=k_n+1$.

\changed{Let $\sSa$ be the structure associated with the function $g$ thus defined.}
Let $n\geq 1$. Orbits of $k_n$-tuples in $\sSa$ are uniquely determined by orbits of $k_n$-tuples in $\sS(\hH_m,g(m))$ for $m\leq n$;
this follows from the fact that $k_n<g(m)$ for $m>n$ and that the orbits of $k$-tuples of $\sS(\hH_m,g(m))$ are that of the empty structure if $k<g(m)$.
By Lemma~\ref{lem:S}, the number of orbits of $k_n$-tuples in $\sSa$ is at most $$3^{k_n^{g(1)|\hH_1|}}\cdots 3^{k_n^{g(n)|\hH_{n}|}} = 3^{\sum k_n^{g(i)|\hH_i|}} < 3^{f(k_n)}.$$
Therefore, the number of orbits of $\sSa$ is bounded above by $3^{f(k)}$ infinitely often.

To prove the final remark, 
{let $f$ be the function given by $f(k) := 3^{\sqrt k}$,
which dominates every polynomial. Let $g \colon {\mathbb N} \to {\mathbb N}$ be the function
obtained from the statement applied to this $f$. Then the orbit growth of $\sSa$ is asymptotically at most $3^{3^{\sqrt{k}}}$, and hence less than doubly exponential. }
\end{proof}

\subsection{Identities in $\sSa$}

\changed{In the present section, $g\colon\mathbb N\to\mathbb N$ is a strictly increasing function, and $\sSa$ denotes the corresponding structure (Construction 5.4).}

\begin{lemma}\label{lem:noh1}
  $\Pol(\sSa)$ does not satisfy any non-trivial height~1 condition.
\end{lemma}
\begin{proof}
  For each non-trivial height~1 condition $\Sigma$ there exists an $n\geq 1$ such that $\Sigma_{\hH_n}$ is weaker than $\Sigma$.
  Since $\CSS(\hH_n)$ does not satisfy $\Sigma_{\hH_n}$, and $\CSS(\hH_n)$ is pp-constructible from $\sS(\hH_n,g(n))$, Theorem~\ref{thm:bop18} implies that the latter does not satisfy $\Sigma_{\hH_n}$ either, and in particular does not satisfy $\Sigma$. 
  Therefore, $\Pol(\sSa)$ does not satisfy~$\Sigma$.
\end{proof}
Since $\Pol(\sSa,\neq)\subseteq\Pol(\sSa)$, we obtain in particular that $\Pol(\sSa,\neq)$ does not satisfy any non-trivial height~1 condition either.

We are not going to show directly that $(\sSa,\neq)$ satisfies non-trivial height~1 conditions locally, but will \changed{find another (not height~1) condition} it satisfies, and then use its slow orbit growth {(for suitable $g$)} to deduce the satisfaction of local height 1 conditions.

A polymorphism $f$ of a~structure is called a~\emph{pseudo-Siggers} operation if there are endomorphisms $e_1,e_2$ of the structure  such that for all $x,y,z$ of the domain
\[
  e_1\circ f(x,y,x,z,y,z) = e_2\circ f(y,x,z,x,z,y).
\]

\begin{lemma}\label{lem:psiggers}
  The structure $(\sSa,\neq)$ has a~pseudo-Siggers polymorphism.
\end{lemma}
\begin{proof}
For each $n\geq 1$, let $\sS_n$ be the generic superposition 
$$
\sS(\hH_1,g(1))\odot \cdots\odot \sS(\hH_n,g(n))\; .
$$ 
Then $\Pol(\sS_n)$ satisfies a~quasi near unanimity identity of some sufficiently large arity. To see this, note that there exists $\ell\geq 1$ such that $\CSS(\hH_1),\ldots, \CSS(\hH_n)$ all have a~quasi near unanimity polymorphism of arity $\ell$, by Lemma~\ref{lem:fg}.
Similarly, such an $\ell$ exists for the $\CSS$-structures on tuples constructed in the proof of Lemma~\ref{lem:S}.
One also sees that in Lemma~\ref{lem:fg}, taking $\ell$ large enough ensures that the constructed polymorphism is also a~polymorphism of $(\CSS(\hH_i),\neq)$.
Thus, the model-complete cores of these structures, i.e., the structures $(\sS(\hH_1,g(1)),\neq)$, \dots, $(\sS(\hH_n,g(n)),\neq)$, also have a~quasi near unanimity polymorphism.
Moreover, these quasi near unanimity polymorphisms have the property that they do not identify any tuples other than those required to be identified by the quasi near unanimity identities.
Hence, since the superposition $\sS_n$ is generic, $(\sS_n,\neq)$ has a~quasi near unanimity polymorphism of arity $\ell$ as well.

By~\cite{BartoPinskerDichotomy, Topo}, it follows that $\Pol(\sS_n,\neq)$ has a~pseudo-Siggers operation for all $n\geq 1$.
Fix, for each $n\geq 1$, a~pseudo-Siggers operation $p_n\in \Pol(\sS_n,\neq)$. We can write 
$$
\Pol(\sSa,\neq)=\bigcap_{n\geq 1} \Pol(\sS_n,\neq)\; .
$$
 By a~standard compactness argument, there exist $\alpha_n \in \Aut(\sSa)$ for all $n\geq 1$ such that the sequence $(\alpha_n\circ p_n)_{n\geq 1}$ converges pointwise to a~function $p$. Clearly, $p\in \Pol(\sSa,\neq)$. 

We finish the proof by showing that $p$ is a~pseudo-Siggers polymorphism of $(\sSa,\neq)$. Let $F$ be a~finite subset of the domain of $\sSa$. Then on $F$ we have $p=\alpha_n\circ p_n$ for almost all $n\geq 1$. By the same argument as for $(\sSa,\neq)$, one sees that each $(\sS_n,\neq)$ is a~model-complete core.
Hence, since $p_n$ is a~pseudo-Siggers polymorphism of \mbox{$(\sS_n,\neq)$}, there exists $\beta_n\in\Aut(\sS_n)$ such that $p_n(x,y,x,z,y,z)=\beta_n\circ p_n(y,x,z,x,z,y)$ for all $x,y,z\in F$. Altogether, we get that for almost all $n\geq 1$ we have that for all $x,y,z\in F$
\begin{align*}
p(x,y,x,z,y,z)&=\alpha_n\circ p_n(x,y,x,z,y,z)\\
&=\alpha_n\circ \beta_n\circ p_n(y,x,z,x,z,y)\\
&=\alpha_n\circ \beta_n\circ 
(\alpha_n)^{-1}\circ p(y,x,z,x,z,y)\; .
\end{align*}
This means that for almost all $n\geq 1$, there exists an automorphism of $\sS_n$ such that $p(x,y,x,z,y,z)$ can be composed with that automorphism from the outside to obtain $p(y,x,z,x,z,y)$ on $F$. By a~standard compactness argument, there exists an automorphism of $\sSa$ with this property. Again by a~standard compactness argument, there exist endomorphisms of $\sSa$ witnessing that $p$ is a~pseudo-Siggers polymorphism of $(\sSa,\neq)$.
\end{proof}

We can therefore apply the following result from~\cite{BKOPP17}.

\begin{theorem}\label{thm:equations}
  Let ${\mathscr C}$ be the polymorphism clone of an $\omega$-categorical model-complete core. Suppose that
  \begin{itemize}
    \item ${\mathscr C}$ satisfies a~non-trivial height 1 identity modulo outer unary functions, and
    \item ${\mathscr C}$ has a~uniformly continuous minion homomorphism to~$\Proj$.
  \end{itemize}
  Then ${\mathscr C}$ has at least doubly exponential orbit growth.
\end{theorem}

\begin{lemma}\label{lem:localh1}
\changed{Suppose that $\sSa$ has less than doubly exponential orbit growth.} Then  there is no uniformly continuous minion homomorphism from $\Pol(\sSa,\neq)$ to $\Proj$.
\end{lemma}
\begin{proof}
\changed{By Lemma~\ref{lem:psiggers}, $\Pol(\sSa,\neq)$ contains a~pseudo-Siggers operation, and by Lemma~\ref{lem:core}, $(\sSa,\neq)$ is an $\omega$-categorical model-complete core.} Hence, Theorem~\ref{thm:equations} implies that $\Pol(\sSa,\neq)$ has no uniformly continuous minion homomorphism to $\Proj$.
\end{proof}

\begin{proof}[Proof of Theorem~\ref{thm:main-3}]
{Let $g \colon {\mathbb N} \to {\mathbb N}$ 
be the function from {the last statement of} Lemma~\ref{lem:orbitgrowth}. 
The structure $(\sSa,\neq)$ from  Construction~\ref{constr:s} is an $\omega$-categorical model-complete core without algebraicity (Lemma~\ref{lem:core})
and has less than doubly exponential orbit growth  (Lemma~\ref{lem:orbitgrowth})}. Moreover, $\Pol(\sSa,\neq)$ has a~minion homomorphism to $\Proj$ by Lemma~\ref{lem:noh1}, but no uniformly continuous such map by Lemma~\ref{lem:localh1}.
\end{proof}

%% file: a_construction.tex
\section{Constructing weaker height~1 conditions}
\label{sect:a1}
Let us consider two loopless graphs ${\rel G}$ and ${\rel H}$ that are not 3-colorable, i.e., $\Sigma_{\rel G}$ and $\Sigma_{\rel H}$ are 
non trivial. An edge $e$ of a~graph ${\rel M}$ is called \emph{critical} if the graph 
\[
{\rel M} - e
\]
obtained from ${\rel M}$ by removing $e$ is 3-colorable. We first replace ${\rel G}$ by a~subgraph of ${\rel G}$ that has a~critical edge $e$. This can be done by repeatedly removing edges until we obtain a~3-colorable graph; the edge that we removed in the last step will be critical for the second-to-last graph of this procedure. Note that the height 1 condition \changed{associated with} the subgraph obtained in this way is still 
non trivial (since the subgraph is not 3-colorable) and implied by the height 1 condition of the original graph by Lemma~\ref{lem:hom-between-sigma_g}.  We modify ${\rel H}$ in the same way as ${\rel G}$, and fix a critical edge $f$ of $\hH$.

Our next step is to glue together $\gG$ and $\hH$ at the critical edges $e$ and $f$ using a gadget graph ${\rel N}$, which is given in Fig.\ \ref{fig:1}. The graph ${\rel N}$ contains four special vertices that are labeled by $x,x',y,y'$, and a special edge labeled by $d$, and has the following properties:
\begin{itemize}
\item Every homomorphism $c\colon {\rel N}\to \rel K_3$ satisfies $c(x) \neq c(x')$ or $c(y) \neq c(y')$ but not both;
\item Every mapping $c\colon \{x,x',y,y'\} \to \rel K_3$ that satisfies the property above extends to a~homomorphism from ${\rel N}$ to~$\rel K_3$.
\item Every mapping $c\colon \{x,x',y,y'\} \to \rel K_3$ that satisfies $c(x) = c(x')$ and $c(y) =c(y')$ can be extended to a~$3$-coloring of~${\rel N}-d$.
\end{itemize}
\begin{figure}
  \[
    \begin{tikzpicture}[ dot/.style = { draw, circle, fill, inner sep=1.5pt }, scale = 2 ]
      \foreach \x in {0,1,3,4} {
        \node [dot] at (\x+.5,.5) {};
        \node [dot] at (\x,0) {};
        \node [dot] at (\x,1) {};
        \draw (\x,0) -- (\x+1,1);
        \draw (\x+1,0) -- (\x,1);
      }
      \node [dot] at (2,0) {};
      \node [dot] at (2,1) {};
      \draw (0,0) -- (5,0);
      \draw (0,1) -- (5,1);

      \node [label=right:{$y'$},dot,fill=white] at (5,0) {};
      \node [label=right:{$y$},dot,fill=white] at (5,1) {};
      \node [label=left:{$x'$},dot,fill=white] at (0,0) {};
      \node [label=left:{$x$},dot,fill=white] at (0,1) {};

      \node [dot] (a) at (2.5,-.5) {};
      \node [dot] (b) at (2.5,1.5) {};

      \draw (2,0) -- (3,1);
      \draw (3,0) -- (2,1);

      \draw (a) edge [bend right=15] node [right] {$d$} (b);
      \draw (2,0) -- (a) -- (3,0);
      \draw (2,1) -- (b) -- (3,1);
    \end{tikzpicture}
  \]
  \caption{The gadget graph ${\rel N}$.} \label{fig:1}
\end{figure}
In our glueing construction, we will only need these three properties of ${\rel N}$, i.e., any other graph with the same properties would work as well. We construct a~new graph, denoted by $({\rel G},e)\oplus ({\rel H},f)$, in the following way:
\begin{enumerate}
  \item We first glue together $\rel N$ and $\gG$ by replacing the edge $e$ by the pair $(x,x')$  of $\rel N$ (the pair $(x,x')$ remaining a~non-edge), and leaving the other vertices disjoint, and then
  \item we add the graph $\hH$ to the construction by replacing the edge $f$ by the pair $(y,y')$ of $\rel N$ (the pair $(y,y')$ remaining a~non-edge).
\end{enumerate}

\begin{lemma} \label{lem:2}
Let ${\rel W} := ({\rel G}, e) \oplus ({\rel H}, f)$ be the graph as constructed above. Then:
  \begin{itemize}
    \item[(1)] ${\rel W}$ is not 3-colorable;
    \item[(2)] The edge $d$ is a critical edge of ${\rel W}$;
    \item[(3)] $\Sigma_{\rel W}$ is implied by both $\Sigma_{\rel G}$ and $\Sigma_{\rel H}$.
  \end{itemize}
\end{lemma}

\begin{proof}
To prove~(1), let us assume that there is a~homomorphism $c\colon {\rel W}\to \rel K_3$. Then neither of its restrictions to the vertices of $\gG$ and the vertices of $\hH$, respectively, is a 3-coloring of ${\rel G}$ or ${\rel H}$, since these graphs are not 3-colorable.  Since all the edges of ${\rel G}$ except $e$ are included in ${\rel W}$, these facts are witnessed on $(x,x')$ and $(y,y')$, i.e., 
         we have that $c(x) = c(x')$ and $c(y) = c(y')$. This implies that the restriction of $c$ to~${\rel N}$ is not a homomorphism (by the properties of $\rel N$ above), a~contradiction.
         
For~(2), we have to show that removing the edge $d$ from ${\rel W}$ we obtain a~3-colorable graph. To find such a~coloring, we first pick 3-colorings of ${\rel G} - e$ and of ${\rel H} - f$, and let $c$ be the union of the two. Then $c$ extends to a~3-coloring of ${\rel W}$, since $c(x) = c(x')$ and $c(y) = c(y')$, and by the properties of $\rel N$ above.

We now prove~(3). Due to the symmetry of the statement it is enough to prove that $\Sigma_{\rel G}$ implies $\Sigma_{\rel W}$. Let us assume that $\clo A$ is a clone which satisfies $\Sigma_{\rel G}$, i.e.,  there are functions $f_v^{\clo A}$ and $g_{(u,v)}^{\clo A}$ for all vertices $v$ of $\gG$ and all edges $(u,v)$ of $\gG$ which witness the satisfaction of $\Sigma_\gG$. We extend this family of functions to a~solution of $\Sigma_{\rel W}$. Before we do that let us fix a $3$-coloring $c$ of the subgraph of ${\rel W}$ induced by the vertices of ${\rel N}$ and ${\rel H}$ such that $c(x) = 1$ and $c(x') = 2$. Such a~coloring exists by the properties of ${\rel H}-f$, of ${\rel N}$, and the construction of $\rel W$. Now, define
    \begin{align*}
      f_1^{\clo A}(x,y,z) & := g_e^{\clo A}(x,y,x,z,y,z)
        \tag{$\spadesuit$.1}\\
      f_2^{\clo A}(x,y,z) & := g_e^{\clo A}(y,x,z,x,z,y)
        \tag{$\spadesuit$.2}\\
      f_3^{\clo A}(x,y,z) & := g_e^{\clo A}(z,z,y,y,x,x)\;.
        \tag{$\spadesuit$.3}
    \end{align*}
    Note that $f_x^{\clo A} = f_1^{\clo A}$ and $f_{x'}^{\clo A} = f_2^{\clo A}$. For any vertex $v$ of ${\rel W}$ which is not a vertex of $\gG$, we put
    \(
      f_v^{\clo A} := f_{c(v)}^{\clo A};
    \)
    for any edge $(u,v)$ of ${\rel W}$ which is not an edge of $\gG$, we define 
    $$
    g_{(u,v)}(x_1,\ldots,x_6):=g_e(x_{\sigma(1)},\ldots,x_{\sigma(6)})\; ,
    $$ 
    where $\sigma$ is a permutation of $\{1,\ldots,6\}$ such that the identities 
    \begin{align*}
      f_u^{\clo A}(x,y,z) &\equals g_{(u,v)}^{\clo A}(x,y,x,z,y,z) \\
      f_v^{\clo A}(x,y,z) &\equals g_{(u,v)}^{\clo A}(y,x,z,x,z,y)
    \end{align*}
    hold.
    This is always possible since when considering any two rows of ($\spadesuit$), the columns of the right-hand side contain all combinations of pairs of different variables.  It is clear that these functions are defined so that they satisfy all identities of $\Sigma_{\rel W}$, which concludes the proof. \qedhere
\end{proof}